\documentclass[12pt]{amsart}

\usepackage{amsmath,amssymb,amscd,xcolor}
\usepackage{fullpage,enumitem}
\usepackage{xypic,units}
\xyoption{curve}

\usepackage{hyperref}
\hypersetup{colorlinks,
	linkcolor=green!50!black,
	citecolor=blue!70!black,
	urlcolor=red!50!black
}

\newcommand\seq{\subseteq}
\newcommand\ssm{\smallsetminus}

\newcommand\sub\subset

\newcommand\A{\ensuremath{\mathbb{A}}}
\newcommand\F{\ensuremath{\mathbb{F}}}
\newcommand\N{\ensuremath{\mathbb{N}}}
\newcommand\Q{\ensuremath{\mathbb{Q}}}
\newcommand\Z{\ensuremath{\mathbb{Z}}}

\newcommand\bmu{{\boldsymbol\mu}}

\newcommand\calD{\mathcal{D}}
\newcommand\calE{\mathcal{E}}
\newcommand\calP{\mathcal{P}}
\newcommand\calS{\mathcal{S}}

\newcommand\dashto\dashrightarrow
\newcommand\longto\longrightarrow

\newcommand\lact\curvearrowright

\newcommand\noi\noindent

\DeclareMathOperator\Aut{Aut}

\DeclareMathOperator\Gal{Gal}

\DeclareMathOperator\genus{genus}
\DeclareMathOperator\id{od}
\DeclareMathOperator\lcm{lcm}
\DeclareMathOperator\ord{ord}
\DeclareMathOperator\rad{rad}

\newcommand\Fp{\F_p}
\newcommand\Fq{\F_q}

\newcommand\Fqn{\F_{q^n}}

\newcommand\Fqbar{\bar\F_q}

\newcommand\Fqtimes{\F_q^\times}

\newcommand\Fell{\F_\ell}
\newcommand\Fellbar{\bar\F_\ell}
\newcommand\Qell{\Q_\ell}
\newcommand\Zell{\Z_\ell}

\newcommand\Flambda{\F_\lambda}
\newcommand\Zlambda{\Z_\lambda}

\newcommand\Fr{\F_r}

\newcommand{\ms}{\medskip}

\newtheorem{theorem}{Theorem}
\newtheorem{prop}[theorem]{Proposition}

\setlist[enumerate]{itemsep=0.05in}

\begin{document}


\counterwithin{equation}{section} 
\numberwithin{theorem}{section}

\renewcommand{\theequation}{\thesection.\arabic{equation}}


\title{Hasse-Weil Zeta Functions Modulo a Prime}
\author{Chris Hall}

\maketitle
\centerline{\today}


\begin{abstract}
Let $\Fq$ be a finite field of characteristic $p$ and $\pi\colon Y\to X$ be a finite $\Fq$-morphism of separated $\Fq$-schemes of finite type.  Suppose $\pi$ is generically Galois with group $G$ of prime order $r\neq p$.  We determine the mod-$r$ reduction of the zeta function of $Y$ in terms of the zeta function of $X$ and the branch locus $Z\sub X$ of $\pi$.  We give applications to curves and to numerators of hyperelliptic/superelliptic curves.
\end{abstract}


\section{Theorem Statement}

Let $p\in\N$ be a prime and $\Fq$ be a finite extension of $\Fp$.
For each separated $\Fq$-scheme $X$ of finite type, let $|X|$ be the set of closed points, that is, the $\Gal(\Fqbar/\Fq)$-orbits of $X(\Fqbar)$, and
$$
	\zeta(X,T)
	:=
	\exp\left(\sum_{n=1}^\infty|X(\Fqn)|\frac{T^n}{n}\right)
	= \prod_{x\in|X|}\det(1-T^{\deg(x)})^{-1}
	\in 1+T\cdot\Z[[T]]
$$
be the (Hasse-Weil) zeta function of $X$.  Recall $\zeta(X,T)\in\Q(T)$ (see \cite{dwork1960rationality} and \cite{grothendieck1966formule}).

\ms
\begin{theorem}
Let $\Fq$ be a finite field of characteristic $p$ and $\pi\colon Y\to X$ be a finite $\Fq$-morphism of separated $\Fq$-schemes of finite type.  Let $Z\sub X$ be the branch locus of $\pi$.  Suppose $\pi$ is generically Galois with group $G$ of prime order $r\neq p$.
Then
\begin{equation}\label{eqn:main-identity}
	\zeta(Y,T)
	\equiv
	\zeta(X,T)^r
	\zeta(Z,T)^{1-r}
	\bmod r.
\end{equation}
\end{theorem}

\noindent
A striking aspect of \eqref{eqn:main-identity} is that the analogous identity over $\Z$ is not necessarily true (because of weight mismatches).

Theorem~\ref{eqn:main-identity} is inspired by \cite[Theorem 4]{hall2006functions}.  We give a naive proof in Section~\ref{sec:naive-proof} and a cohomological proof in Section~\ref{sec:cohomological-proof}.  The latter uses a cohomological formula given in Section~\ref{sec:zeta-modulo-ell} for the reduction of $\zeta(X,T)$ modulo a prime $\ell\neq p$.  We also give two corollaries (for curves) in Section~\ref{sec:two-corollaries}.

An impetus for this paper is question of Richard Griffon: when is the numerator of the zeta function of a curve over a finite field a trinomial?  A necessary condition is that the numerator is congruent to a trinomial modulo an integer $N>1$.
In Section~\ref{sec:main-application}, we apply the results of Section~\ref{sec:two-corollaries} (where $N=r$) to numerators of hyperelliptic ($r=2$) and superelliptic ($r>2$) curves.


\vfil
\section{Naive Proof of Theorem}\label{sec:naive-proof}

Consider the Euler-product expansion
\begin{equation}\label{eqn:euler-product-expansion}
	\zeta(Y,T)
	=
	\prod_{y\in|Y|}(1-T^{\deg(y)})^{-1} 
	=
	\prod_{x\in|X|}\prod_{y\in|Y|:\pi(y)=x}(1-T^{\deg(y)})^{-1}.
\end{equation}
Suppose $x\in|X|$, and let
\begin{equation}\label{eqn:defn-of-Y_x-and-L(Y_x,T)}
	|Y_x|:=\{y\in|Y|:\pi(y)=x\}\text{ and }
	L(Y_x,T):=\prod_{y\in |Y_x|}(1-T^{\deg(y)}).
\end{equation}
Suppose $y\in|Y_x|$, and let $I_y\seq D_y\seq G$ be the inertia and decomposition groups of $y$.

Recall that $\pi$ is \'etale over $X\ssm Z$ and totally ramified of prime degree $r$ over $Z$ and that $G\cong\Z/r$ acts transitively on $|Y_x|$.  Deduce
\begin{equation}\label{eqn:triple-classification}
	(|I_y|,[D_y:I_y],|Y_{x}|)
	=
	\begin{cases}
		(1,1,r)\text{ or }(1,r,1) & x\in |X|\ssm |Z| \\
		(r,1,1) & x\in |Z|
	\end{cases}.
\end{equation}
Observe that
$$
	(1-T^{\deg(x)})^r\equiv (1-T^{r\deg(x)})\bmod r,
$$
and \eqref{eqn:triple-classification} imply
\begin{equation}\label{eqn:euler-congruence}
	L(Y_x,T)
	\equiv
	\begin{cases}
		(1-T^{\deg(x)})^r & x\in |X|\ssm |Z| \\
		1-T^{\deg(x)} & x\in |Z|
	\end{cases}.
\end{equation}
Combining \eqref{eqn:euler-product-expansion}, \eqref{eqn:defn-of-Y_x-and-L(Y_x,T)}, and \eqref{eqn:euler-congruence} gives
$$
	\zeta(Y,T)
	\equiv
	\prod_{x\in|X|}(1-T^{\deg(x)})^{-r}\prod_{z\in|Z|}(1-T^{\deg(z)})^{r-1}
	\equiv
	\zeta(X,T)^r \zeta(Z,T)^{1-r}
	\bmod r
$$
as desired.


\vfil
\section{Cohomological Reduction of Zeta Modulo a Prime }\label{sec:zeta-modulo-ell}

\begin{theorem}\label{thm:zeta-modulo-ell-via-cohomology}
Let $X$ be a separated $\Fq$-scheme of finite type and $\bar{X}$ be its base change to $\Fqbar$.
Let $\phi\in\Gal(\Fqbar/\Fq)$ be the geometric Frobenius and $\ell\neq p$ be a rational prime.  Then
\begin{equation}\label{eqn:zeta-mod-ell-as-alternating-product}
	\zeta(X,T)
	\equiv
	\prod_i\det(1-\phi\,T \mid H^i_c(\bar{X},\Fell))^{(-1)^{i+1}} \bmod\ell.
\end{equation}
\end{theorem}

\begin{proof}
A theorem of Grothendieck (see \cite[Th\'eor\`eme 5.1]{grothendieck1966formule} or \cite[Theorem~10.5.1]{fu2011etale}) implies that
$$
	\zeta(X,T)=\prod_i\det(1-\phi\,T \mid H^i_c(\bar{X},\Qell))^{(-1)^{i+1}}
$$
(since $X$ is $\Fq$-compactifiable).  We must analyze a product with $\Fell$ in lieu of $\Qell$.  This requires taking into account that $M_i:=H^i_c(\bar{X},\Zell)$ is a finitely generated (see \cite[Expos\'ee VI, Lemme 2.2.1]{sga5}) and not necessarily free (compare \cite{MR725400}).

Let $T_i\seq  M_i$ be the torsion submodule and $F_i:=M_i/T_i$ be the maximal free quotient so that
$$
	\det(1 - \phi\,T \mid H^i_c(\bar{X},\Qell))
	=
	\det(1 - \phi\,T \mid F_i)
	\in 1+T\cdot\Zell[T]
$$
and thus
\begin{equation}\label{eqn:zeta-over-Z_ell-as-alternating-product}
	\zeta(X,T)=\prod_i\det(1-\phi\,T \mid F_i)^{(-1)^{i+1}}.
\end{equation}

Consider the exact sequence of \'etale $\Zell$-sheaves on $X$ given by
$$
	0\longto\Zell\overset{\times\ell}\longto\Zell\longto\Fell\longto 0.
$$
Its long exact cohomology sequence breaks into short exact sequences of $\Fell[\phi]$-modules
\begin{equation*}
	0
	\longto H^i_c(\bar{X},\Zell)\otimes\Fell
	\longto H^i_c(\bar{X},\Fell)
	\longto H^{i+1}_c(\bar{X},\Zell)[\ell]
	\longto 0
\end{equation*}
or equivalently
\begin{equation}\label{lem:short-exact-with-F_ell-cohomology}
	0
	\longto (T_i\oplus F_i)\otimes\Fell
	\longto H^i_c(\bar{X},\Fell)
	\longto T_{i+1}[\ell]
	\longto 0.
\end{equation}

For each finite $\Fell[\phi]$-module $M$, let
$$
	\Lambda(M):=\det(1-\phi\,T\mid M)\in 1+T\cdot \Fell[T].
$$  
Observe that the exactness of \eqref{lem:short-exact-with-F_ell-cohomology} implies that $$
	\Lambda(H^i_c(\bar{X},\Fell))
	=
	\Lambda((T_i\oplus F_i)\otimes\Fell)\cdot
	\Lambda(T_{i+1}[\ell])
	=
	\Lambda(F_i\otimes\Fell)\cdot
	\Lambda(T_i\otimes\Fell)\cdot
	\Lambda(T_{i+1}[\ell]).
$$
Moreover, Proposition~\ref{prop:ell-torsion-vs-tensor-with-F_ell} (in Appendix~\ref{sec:appendix}) implies that
$$
	\Lambda(T_{i}\otimes\Fell)
	=
	\Lambda(T_{i}[\ell]).
$$

Deduce that
$$
	\prod_i(\Lambda(T_i\otimes\Fell)\cdot\Lambda(T_{i+1}[\ell]))^{(-1)^{i+1}}
	=
	\prod_i\Lambda(T_i[\ell])^{(-1)^{i+1}}\cdot\prod_i\Lambda(T_{i+1}[\ell])^{(-1)^{i+1}}
	=
	1	
$$
(since the product telescopes) and 
\begin{equation}\label{eqn:comparing-F_ell-and-Z_ell-parts}
	\prod_i\Lambda(H^i_c(\bar{U},\Fell))^{(-1)^{i+1}}
	=
	\prod_i\Lambda(F_i\otimes\Fell)^{(-1)^{i+1}}
	\equiv
	\prod_i\det(1-\phi\,T\mid F_i)^{(-1)^{i+1}}
	\bmod\ell.
\end{equation}
The desired identity \eqref{eqn:zeta-mod-ell-as-alternating-product} now follows from \eqref{eqn:zeta-over-Z_ell-as-alternating-product} and \eqref{eqn:comparing-F_ell-and-Z_ell-parts}.
\end{proof}


\subsection{Appendix: Linear Algebra}\label{sec:appendix}

\begin{prop}\label{prop:ell-torsion-vs-tensor-with-F_ell}
Suppose $\Zlambda$ is a finite extension of $\Zell$ and $A$ is a finite $\Zlambda[\phi]$-module.  Let $\lambda\in\Zlambda$ be a uniformizer and $\Flambda:=\Zlambda/\lambda\Zlambda$ be the residue field.  Let $A[\lambda]\seq A$ be the $\Zlambda[\phi]$-submodule annihilated by $\lambda$.  Then $A[\lambda]$ and $A\otimes\Flambda:=A/\lambda A$ are finite $\Flambda[\phi]$-modules, and
\begin{equation}\label{eqn:A[lambda]-vs-A/lambdaA}
	\det(1-\phi\,T\mid A[\lambda])
	=
	\det(1-\phi\,T\mid A\otimes\Flambda).
\end{equation}
\end{prop}

\begin{proof}
Suppose $n\geq 0$ and $\lambda^n\Zlambda\seq\Zlambda$ is the annihilator of $A$.

Observe that the proposition is true for $n=0$ since $A[\lambda]=A\otimes\Flambda=0$.

Suppose that $n>0$ and that
\begin{equation}\label{eqn:inductive-identity}
	\det(1 - \phi\,T\mid B[\lambda])=\det(1 - \phi\,T\mid B\otimes\Flambda)
\end{equation}
for every finite $\Zlambda[\phi]$-module $B$ annihilated by $\lambda^{n-1}\Zlambda$.

Observe that $B:=\lambda A\seq A$ and $\lambda B\seq B$ are $\Zlambda[\phi]$-invariant.  Moreover, \eqref{eqn:inductive-identity} holds since $\lambda^{n-1}\Zlambda$ annihilates $B$.

Consider the commutative diagram of $\Zlambda[\phi]$-modules with injective/surjective morphisms and exact rows/columns
$$
	\xymatrix{
		B[\lambda]\ar@{^{(}->}[r]\ar@{^{(}->}[d]
		& A[\lambda]\ar@{->>}[r]\ar@{^{(}->}[d]\ar@{-->}[dr]
		& A[\lambda]/B[\lambda]\ar@{^{(}->}[d]
		\\
		B\ar@{^{(}->}[r]\ar^{\times\lambda}@{->>}[d]
		& A\ar@{->>}[r]\ar^{\times\lambda}@{->>}[d]
		& A\otimes\Flambda\ar^{\times\lambda}@{->>}[d]
		\\
		\lambda B\ar@{^{(}->}[r]
		& B\ar@{->>}[r]
		& B\otimes\Flambda
	}.
$$
Applying the snake lemma to the bottom two rows yields
\begin{equation}\label{eqn:result-of-snake-lemma}
	0\longto B[\lambda]\longto A[\lambda]\longto A\otimes\Flambda\longto B\otimes\Flambda\longto 0
\end{equation}
is an exact sequence of $\F_\lambda[\phi]$-modules.  In particular, \eqref{eqn:inductive-identity} and \eqref{eqn:result-of-snake-lemma} imply that
\begin{equation*}\label{eqn:herbrand-quotients-wrt-A-and-B}
	\frac
	{\det(1 - \phi\,T\mid A\otimes\Flambda)}
	{\det(1 - \phi\,T\mid A[\lambda])}
	=
	\frac
	{\det(1 - \phi\,T\mid B[\lambda])\cdot\det(1 - \phi\,T\mid A\otimes\Flambda)}
	{\det(1 - \phi\,T\mid A[\lambda])\cdot \det(1 - \phi\,T\mid B\otimes\Flambda)}
	=
	1
\end{equation*}
as desired.
\end{proof}


\vfil
\section{Cohomological Proof of Theorem}\label{sec:cohomological-proof}

Let $j\colon U\to X$ be the inclusion of $U:=X\ssm Z$ and $V:=\pi^{-1}(U)$.  Observe that
\begin{equation}\label{eqn:Z(X,T)-as-product}
		\zeta(X,T)=\zeta(U,T)\zeta(Z,T)
\end{equation}
since $X=U\sqcup Z$ and that
\begin{equation}\label{eqn:Z(Y,T)-as-product}
	\zeta(Y,T)=\zeta(V,T)\zeta(Z,T)
\end{equation}
since $\pi$ induces an isomorphism $\pi^{-1}(Z)\to Z$.

Let $\phi\in\Gal(\Fqbar/\Fq)$ be the geometric Frobenius.  Recall that $r:=|G|$ is prime and $r\neq p$.  Observe that $\calP_r:=\pi_*\Fr$ is a constructible sheaf of $\Fr[G]$-modules and that
$$
	H^i_c(\bar{V},\Fr)\cong H^i_c(\bar{U},\calP_r)
$$
as $\Fr[\phi]$-modules since $\pi$ is finite (see \cite[Prop.~5.7.4]{fu2011etale}), hence
\begin{equation}\label{eqn:zeta(V,T)-mod-r}
	\zeta(V,T) \equiv \prod_i\det(1-\phi\,T\mid H^i_c(\bar{U},\calP_r))\bmod r
\end{equation}
by Theorem~\ref{thm:zeta-modulo-ell-via-cohomology}.

Observe that $j^*\calP_r$ is a lisse $\Fr[G]$-sheaf on $U$ of rank $r$ and that its $G$-semisimplification $\calS_r$ satisfies
\begin{equation}\label{eqn:passage-to-semisimplification}
	\det(1-\phi\,T\mid H^i_c(\bar{U},\calP_r))
	=
	\det(1-\phi\,T\mid H^i_c(\bar{U},\calS_r))
\end{equation}
for every $i$.  Also, $\Fr$ is the only simple $\Fr[G]$-module, hence $\calS_r\cong\Fr^r$ as $\Fr[G]$-sheaves and
\begin{equation}\label{eqn:passage-to-power-of-trivial}
	\det(1-\phi\,T\mid H^i_c(\bar{U},\calS_r))
	=	
	\det(1-\phi\,T\mid H^i_c(\bar{U},\Fr))^r
\end{equation}
for every $i$.  In particular, combining \eqref{eqn:zeta(V,T)-mod-r}, \eqref{eqn:passage-to-semisimplification}, and \eqref{eqn:passage-to-power-of-trivial} yields
$$
	\zeta(V,T)
	\equiv
	\prod_i\det(1-\phi\,T\mid H^i_c(\bar{U},\Fr))^{r(-1)^{i+1}}
	\equiv
	\zeta(U,T)^r\bmod r
$$
which combines with \eqref{eqn:Z(X,T)-as-product} and \eqref{eqn:Z(Y,T)-as-product} to yield
\begin{equation}\label{eqn:main-identity-restated}
	\zeta(Y,T)
	\equiv
	\zeta(U,T)^r\zeta(Z,T)
	\equiv
	\zeta(X,T)^r\zeta(Z,T)^{1-r}
	\bmod r
\end{equation}
as desired.


\vfil
\section{Corollaries for Curves}\label{sec:two-corollaries}

Suppose $X,Y$ are proper, smooth, and geometrically connected curves over $\Fq$.
Let
$$
	L(C,T):=\zeta(X,T)(1-T)(1-qT)\in 1+T\Z[T]
$$
be the numerator of $\zeta(C,T)$ for each $C\in\{X,Y\}$, and recall
$$
	\deg(L(C,T))=2\cdot\genus(C).
$$
Observe that \eqref{eqn:main-identity} (restated in \eqref{eqn:main-identity-restated})  is equivalent to
\begin{equation}\label{eqn:L(Y,T)-modulo-r}
	L(Y,T)\equiv L(X,T)^r(1-T)^{1-r}(1-qT)^{1-r}\prod_{z\in|Z|}(1-T^{\deg(z)})^{r-1}\bmod r.
\end{equation}
Moreover, taking degrees of both sides of \eqref{eqn:main-identity} yields the Riemann-Hurwitz formula:
\begin{equation}\label{eqn:riemann-hurwitz}
	2\cdot\genus(Y) - 2 = r\cdot(2\cdot\genus(X)-2)+(r-1)\cdot\deg(Z).
\end{equation}


\vfil
\section{Application to Hyperelliptic/Superelliptic Curves}\label{sec:main-application}

Let $\Fq$ be a finite field, $\pi_1,\ldots,\pi_d\in\Fq[x]$ be distinct monic irreducibles, $e_1,\ldots,e_d\in\N$ be positive, and
$$
	f:=\prod_{i=1}^m\pi_i^{e_i},\ \rad(f):=\prod_{i=1}^m\pi_i\in\Fq[x].
$$
Let $Z\sub\A^1$ be the zero locus of $f$, and observe $\deg(Z)=\deg(\rad(f))$.

Let $r$ be a prime divisor of $q-1$ and $A/\Fq$ be the affine curve $y^r=f(x)$.  The group $G:=\bmu_r\seq\Fqtimes$ acts faithfully on $A$ via $\zeta(x,y)=(x,\zeta y)$ for $\zeta\in G$, and the morphism $A\to\A^1$ given by $(x,y)\mapsto x$ is the quotient map $A\to A/G$.

Let $A^\nu\to A$ be the normalization of $A$ and $A^\nu\to Y$ be the smooth completion of $A^\nu$.  The action $G\lact A$ extends uniquely to (faithful) actions $G\lact A^\nu$ and $G\lact Y$ such that the morphisms 
$A^\nu\to A$ and $A^\nu\to Y$ are $G$-equivariant (see Propositions~\ref{prop:normalization-and-automorphisms} and \ref{prop:smooth-completion-and-automorphisms} in Appendix~\ref{sec:group-actions}).  The composed morphism $A^\nu\to A\to\A^1$ is the quotient $A^\nu\to A^\nu/G$.  It extends uniquely to a morphism $\pi\colon Y\to\P^1$, the quotient $Y\to Y/G$.

Suppose $f$ that is $r$th power free, that is, $0<e_i<r$ for $1\leq i\leq m$ so that $N\to A$ is bijective (on points).  Suppose that $r\mid\deg(f)$ so that $Z\sub\P^1$ is the ramification locus of $\pi\colon Y\to\P^1$.  Then
\begin{equation}\label{eqn:superelliptic-L-polynomial-modulo-r}
	L(Y,T)\equiv (1-T)^{2-2r}\prod_{i=1}^m(1-T^{\deg(\pi_i)})^{r-1}\bmod r
\end{equation}
by \eqref{eqn:L(Y,T)-modulo-r} since $L(\P^1,T)=1$ and $q\equiv 1\bmod r$.  Moreover,
\begin{equation}\label{eqn:superelliptic-genus-formula}
	2\cdot\genus(Y)
	=
	2 + r(-2)+(r-1)\deg(Z)
	=
	(r-1)(\deg(\rad(f))-2)
\end{equation}
by \eqref{eqn:riemann-hurwitz}.


\subsection{Trinomial Hyperelliptics}
 
Suppose $r=2$ (hence $q$ is odd), and let $g:=\genus(Y)$.  Observe that
\begin{equation}\label{eqn:product-of-cyclotomics}
	(1-T)^2L(Y,T)
	\equiv
	\prod_{i=1}^m(1-T^{\deg(\pi_i)})
	\equiv
	\zeta(Z,T)^{-1}
	\bmod 2
\end{equation}
by \eqref{eqn:superelliptic-L-polynomial-modulo-r} (compare \cite[Prop.~2.10]{MR4427965}) and
$$
	\deg(f) = \deg(\rad(f)) = 2g + 2
$$
by \eqref{eqn:superelliptic-genus-formula}.

\begin{prop}\label{prop:Z(Z,T)^{-1}-modulo-ell}
Let $\ell\in\N$ be a prime and $\calD$ be a finite submultiset of $\N_{\ell'}:=\N\ssm \ell\N$.
Let $Z$ be a finite $\Fq$-scheme and $Z_n$ be its base change to $\Fqn$.
Let
$$
	e:=\lcm(\ord_\ell(\deg(z):z\in|Z|))
$$
and $d:=\ell^e$.  Then
\begin{enumerate}
\item\label{item:deg(z_d)-is-coprime-to-ell} $\deg(z_d)\in\N_{\ell'}$ for every $z_d\in|Z_d|$;
\item\label{item:detecting-Z(Z,T)} $\zeta(Z,T)^{-1}\equiv\prod_{d\in\calD}(1-T^d)\bmod \ell$ if and only if $\calD=\{\deg(z_d):z_d\in|Z_d|\}$ as multisets;
\item\label{item:ord-T=1-of-Z(Z,T)} $\ord_{T=1}(\zeta(Z,T)^{-1}\in\F_\ell[T])=||Z_d||$.
\end{enumerate}
\end{prop}

\begin{proof}
If $z\in|Z|$ and $e':=\ord_\ell(\deg(z))$, then $z$ splits into $\ell^{e'}$ points in $Z_{\ell^e}$, all of degree $\frac{1}{\ell^{e'}}\deg(z)\in\N_{\ell'}$, so \ref{item:deg(z_d)-is-coprime-to-ell} holds.  
Proposition~\ref{prop:degree-ell-base-change-leaves-mod-ell-reduction-unchanged} (in Appendix~\ref{sec:scalar-base-change}) and Proposition~\ref{prop:common-cyclotomic-factorizations} (in Appendix~\ref{sec:cyclotomic-factorization})
then imply that part \ref{item:detecting-Z(Z,T)} holds.  Finally, $\ord_{T=1}(1-T^m)=1$ for $m\in\N_{\ell'}$, hence part \ref{item:ord-T=1-of-Z(Z,T)} holds.
\end{proof}

We give necessary and sufficient conditions for
\begin{equation}\label{eqn:trinomial-condition}
	L(Y,T)
	\equiv
	1 + aT^g + q^gT^{2g}\in\F_2[T].
\end{equation}

\ms\noi
\underline{$a=0$}:
Let $2g=o2^e$ be the unique factorization in $\Z$ with $o$ odd.  Observe $e\geq 1$ and \eqref{eqn:trinomial-condition} is equivalent to
$$
	(1+T)^2L(Y,T)
	\equiv
	(1+T)^2(1+T^{2g})
	\equiv
	(1+T)^2(1+T^o)^{2^e}\bmod 2.
$$
This holds (by Proposition~\ref{prop:Z(Z,T)^{-1}-modulo-ell}) if and only if $f$ splits over $\F_{q^{2^e}}$ into a product of two linears and $2^e$ irreducibles of (odd) degree $o$.

\ms\noi
\underline{$a=1$}:
Let $g=o2^e$ be the unique factorization in $\Z$ with $o$ odd.  Observe $e\geq 0$  and \eqref{eqn:trinomial-condition} is equivalent to
\begin{equation}\label{eqn:L(Y,T)-mod-2-for-a=1}
	(1+T)^2L(Y,T)
	\equiv
	(1+T)^2(1+T^g+T^{2g})
	\equiv
	(1+T)^2(1+T^o+T^{2o})^{2^e}\bmod 2.
\end{equation}
Moreover, $\ord_{T=1}(1+T^o+T^{2o})=0$ since $o$ is odd.  Therefore \eqref{eqn:product-of-cyclotomics} and 
Proposition~\ref{prop:Z(Z,T)^{-1}-modulo-ell} imply that \eqref{eqn:L(Y,T)-mod-2-for-a=1} holds if and only if $f$ factors over $\F_{q^{2^e}}$ as a product of two irreducibles of odd degrees $d_1\leq d_2$ and
\begin{eqnarray*}
	(1+T^2)(1+T^g+T^{2g})
	& \equiv & 	1+T^2+T^g+T^{g+2}+T^{2g-2}+T^{2g} \\
	&\equiv &	1+T^{d_1}+T^{d_2}+T^{d_1+d_2}\bmod 2.
\end{eqnarray*}
The last equivalence holds if and only if $(g,e,d_1,d_2)$ equals $(1,0,1,3)$ or $(2,1,3,3)$.


\vfil
\subsection{Appendix: Groups Actions}\label{sec:group-actions}

\begin{prop}\label{prop:normalization-and-automorphisms}
Let $S$ be an integral scheme and $\nu\colon S^\nu\to S$ be a normalization morphism.  There is a unique homomorphism $\nu^*\colon\Aut(S)\to\Aut(S^\nu)$ given by $\alpha\mapsto\alpha^\nu$ and making
\begin{equation}\label{eqn:normalization-square}
	\xymatrix{
		S^\nu\ar[d]_\nu & S^\nu\ar[l]_{\alpha^\nu}\ar[d]^\nu \\
		S & S\ar[l]_{\alpha}
	}
\end{equation}
commute for every $\alpha\in\Aut(S)$.
\end{prop}

\begin{proof}
By definition of a normalization morphism (see \cite[\S4.1.2, Def.~2.19]{liu2002algebraic}), $S^\nu$ is normal and every dominant morphism $t\colon T\to S$ with $T$ normal factors uniquely through $\nu$:
$$
	\xymatrix{
		S^\nu\ar[d]_\nu & T\ar[l]_{t^\nu}\ar[dl]^t \\
		S
	}
$$
In particular, if $\alpha\in\Aut_k(S)$ and $t:=\alpha\circ\nu\colon S^\nu\to S$, then
$\alpha^\nu:=t^\nu\colon S^\nu\to S^\nu$ is the unique morphism making \eqref{eqn:normalization-square} commute.
Uniquenesses forces $\id_S^\nu=\id_{S^\nu}$ and $(\beta^{-1})^\nu\circ\alpha^\nu=(\beta^{-1}\circ\alpha)^\nu$ for each $\alpha,\beta\in\Aut(S)$.  Also, if $\alpha\in\Aut(S)$, then $\alpha^\nu\in\Aut(S^\nu)$ since $(\alpha^{-1})^\nu\circ\alpha^\nu=\id_{S^\nu}=\alpha^\nu\circ(\alpha^{-1})^\nu$.  Therefore $\alpha\mapsto\alpha^\nu$ gives a group homomorphism $\Aut(S)\to\Aut(S^\nu)$.
\end{proof}

\begin{proof}
\end{proof}

\begin{prop}\label{prop:smooth-completion-and-automorphisms}
Let $k$ be a field and $S$ be a smooth connected $k$-curve.  Let $S\to S^\kappa$ be the smooth completion of $S^\nu$.  There is a unique homomorphism $\kappa^*\colon\Aut_k(S)\to\Aut_k(S^\kappa)$ given by $\alpha\mapsto\alpha^\kappa$ and making
\begin{equation}
	\xymatrix{
		S^\kappa\ar[d]_\kappa & S^\kappa\ar[l]_{\alpha^\kappa}\ar[d]^\kappa \\
		S & S\ar[l]_{\alpha}
	}
\end{equation}
commute for every $\alpha\in\Aut_k(S)$.
\end{prop}

\begin{proof}
Suppose $\alpha,\beta\in\Aut_k(S)$.  Observe that $\alpha$ represents a unique birational $k$-map $[\alpha]\colon S^\kappa\dashto S^\kappa$ since $S,S^\kappa$ are birational.  Moreover, $[\alpha]$ extends uniquely to a $k$-morphism $\alpha^\kappa\in S^\kappa\to S^\kappa$ since $S^\kappa$ is smooth and complete (see \cite[Ch.~2, Prop.~6.8]{hartshorne1977algebraic} or \cite[\href{https://stacks.math.columbia.edu/tag/0BXZ}{Lemma 0BXZ}]{stacks-project}).
Uniquenesses forces $\id_S^\kappa=\id_{S^\kappa}$ as well as $(\beta^{-1})^\kappa\circ\alpha^\kappa=(\beta^{-1}\circ\alpha)^\kappa$ and $\alpha^\kappa\in\Aut_k(S^\kappa)$.  Therefore
 $\alpha\mapsto\alpha^\kappa$ gives a group homomorphism $\Aut_k(S)\to\Aut_k(S^\kappa)$.
\end{proof}


\vfil
\subsection{Appendix: Scalar Base Change}\label{sec:scalar-base-change}

\begin{prop}\label{prop:degree-ell-base-change-leaves-mod-ell-reduction-unchanged}
Let $\Fq$ be a finite field characteristic $p$ and $\ell\neq p$ be a rational prime.
Let $S$ be a separated $\Fq$-scheme of finite type and $S_n$ be its base change to $\F_{q^n}$.  Then $S=S_1$ and
$$
	\zeta(S_\ell,T)\equiv\zeta(S_1,T)\bmod\ell.
$$
\end{prop}

\begin{proof}
Let $\phi\in\Gal(\Fqbar/\Fq)$ be the geometric Frobenius and
$$
	\zeta_i(S_n,T) := \det(1 - \phi^n\,T\mid H^i_c(\bar{S},\Fell))\in 1+T\cdot\Fell[T]
$$
so that Theorem~\ref{thm:zeta-modulo-ell-via-cohomology} implies
\begin{equation}\label{eqn:base-change-zeta-as-product}
	\zeta(S_n,T)
	\equiv
	\prod_i\zeta_i(S_n,T)^{(-1)^{i+1}}\bmod\ell.
\end{equation}
Let $d_i:=\deg(\zeta_i(S_1,T))$ and
$$
	\zeta_i(S_1,T)=\prod_{j=1}^{d_i}(1-\alpha_{i,j}T)\in\Fellbar[T]
$$
be a factorization.  Observe that
\begin{equation}\label{eqn:congruence-of-factors}
	\zeta_i(S_\ell,T)
	\equiv \prod_{j=1}^{d_i}(1-\alpha_{i,j}^\ell T)
	\equiv \prod_{j=1}^{d_i}(1-\alpha_{i,j} T)
	\equiv \zeta_i(S_1,T)
	\bmod\ell
\end{equation}
since $\psi_i(S_\ell,T)\in\Fell[T]$.  Then \eqref{eqn:base-change-zeta-as-product} and \eqref{eqn:congruence-of-factors} imply
$$
	\zeta(S_\ell,T)
	\equiv
	\prod_i\zeta_i(S_\ell,T)^{(-1)^{i+1}}
	\equiv
	\prod_i\zeta_i(S_1,T)^{(-1)^{i+1}}
	\equiv
	\zeta(S_1,T)
	\bmod\ell
$$
as desired.
\end{proof}


\vfil
\subsection{Appendix: Cyclotomic Factorization}\label{sec:cyclotomic-factorization}

\begin{prop}\label{prop:common-cyclotomic-factorizations}
Let $\ell\in\N$ be a prime, $\calD,\calE$ be finite submultisets of $\N_{\ell'}:=\N\ssm\ell\N$, and
$$
	\psi_\calD:=\prod_{d\in\calD}(1-T^{d})
	\text{ and }
	\psi_{\calE}:=\prod_{e\in\calE}(1-T^{e})
$$
If $\psi_\calD=\psi_E$, then $\calD=\calE$ as multisets.
\end{prop}

\begin{proof}
We induct on $|\calD|,|\calE|$.
For the base case, observe $\calD=\emptyset$ if and only if $\calE=\emptyset$.
Suppose that $\calD,\calE$ are both nonvoid and that $\psi_{\calD}=\psi_{\calE}$.  We show that $d:=\max(\calD)$ equals $e:=\max(\calE)$ and induct.

Let $\psi_n\in\F_r[T]$ be the cyclotomic polynomial
$$
	\psi_n:=\prod_{m\mid n}(1-T^m)^{\mu(n/m)}.
$$
Observe that $\psi_m,\psi_n$ are coprime when $m,n\in\N_{\ell'}$ are distinct, hence
$$
	d=\max\{n\in\N_{r'}:\gcd(\psi_n,\psi_{\calD})\neq 1\}.
$$
Deduce $d\geq\max(\calE)$ since $\psi_d\mid\psi_{\calE}=\psi_{\calD}$ for $e\in\calE$.
A symmetric argument implies that 
$$
	e=\max(\calE)\geq\max(\calD)=d,
$$
hence $d=e$.

Let $\calD':=\calD\ssm\{d\}$ and $\calE':=\calE\ssm\{e\}$.
Observe that the identities
$$
	\psi_{\calD'}\cdot(1-T^d)
	=
	\psi_{\calD}
	=
	\psi_{\calE}
	=
	\psi_{\calE'}\cdot(1-T^e)
$$
imply $\psi_{\calD'}=\psi_{\calE'}$.  Deduce $\calD'=\calE'$ as multisets (by induction) and $\calD=\calE$ as multisets.
\end{proof}


\vfil
\bibliography{zeta-functions-modulo-ell.bib}
\bibliographystyle{plain}


\end{document}